\documentclass[12pt]{amsart}
\usepackage{amsmath,amssymb,amsbsy,amsfonts,latexsym,amsopn,amstext,cite,
                                               amsxtra,euscript,amscd,bm,mathabx}
\usepackage{url}
\usepackage[colorlinks,linkcolor=blue,anchorcolor=blue,citecolor=blue,backref=page]{hyperref}
\usepackage{color}
\usepackage{graphics,epsfig}
\usepackage{graphicx}
\usepackage{float} 
\usepackage[english]{babel}
\usepackage{mathtools}
\usepackage{todonotes}
\usepackage{url}
\usepackage[colorlinks,linkcolor=blue,anchorcolor=blue,citecolor=blue,backref=page]{hyperref}

\usepackage[norefs,nocites]{refcheck}

\hypersetup{breaklinks=true}

\usepackage[norefs,nocites]{refcheck}
\usepackage[english]{babel}
\begin{document}

\newtheorem{thm}{Theorem}
\newtheorem{lem}[thm]{Lemma}
\newtheorem{claim}[thm]{Claim}
\newtheorem{cor}[thm]{Corollary}
\newtheorem{prop}[thm]{Proposition} 
\newtheorem{definition}[thm]{Definition}
\newtheorem{rem}[thm]{Remark} 
\newtheorem{question}[thm]{Open Question}
\newtheorem{conj}[thm]{Conjecture}
\newtheorem{prob}{Problem}

\newtheorem{lemma}[thm]{Lemma}

\newcommand{\GL}{\operatorname{GL}}
\newcommand{\SL}{\operatorname{SL}}
\newcommand{\lcm}{\operatorname{lcm}}
\newcommand{\ord}{\operatorname{ord}}
\newcommand{\Op}{\operatorname{Op}}
\newcommand{\Tr}{\operatorname{Tr}}
\newcommand{\Nm}{\operatorname{Nm}}

\numberwithin{equation}{section}
\numberwithin{thm}{section}
\numberwithin{table}{section}

\def\vol {{\mathrm{vol\,}}}
\def\squareforqed{\hbox{\rlap{$\sqcap$}$\sqcup$}}
\def\qed{\ifmmode\squareforqed\else{\unskip\nobreak\hfil
\penalty50\hskip1em\null\nobreak\hfil\squareforqed
\parfillskip=0pt\finalhyphendemerits=0\endgraf}\fi}

\def \balpha{\bm{\alpha}}
\def \bbeta{\bm{\beta}}
\def \bgamma{\bm{\gamma}}
\def \blambda{\bm{\lambda}}
\def \bchi{\bm{\chi}}
\def \bphi{\bm{\varphi}}
\def \bpsi{\bm{\psi}}
\def \bomega{\bm{\omega}}
\def \btheta{\bm{\vartheta}}

\newcommand{\bfxi}{{\boldsymbol{\xi}}}
\newcommand{\bfrho}{{\boldsymbol{\rho}}}

\def\Kab{\sfK_\psi(a,b)}
\def\Kuv{\sfK_\psi(u,v)}
\def\SaUV{\cS_\psi(\balpha;\cU,\cV)}
\def\SaAV{\cS_\psi(\balpha;\cA,\cV)}

\def\SUV{\cS_\psi(\cU,\cV)}
\def\SAB{\cS_\psi(\cA,\cB)}

\def\Kmnp{\sfK_p(m,n)}

\def\KKap{\cH_p(a)}
\def\KKaq{\cH_q(a)}
\def\KKmnp{\cH_p(m,n)}
\def\KKmnq{\cH_q(m,n)}

\def\Klmnp{\sfK_p(\ell, m,n)}
\def\Klmnq{\sfK_q(\ell, m,n)}

\def \SALMNq {\cS_q(\balpha;\cL,\cI,\cJ)}
\def \SALMNp {\cS_p(\balpha;\cL,\cI,\cJ)}

\def \SACXMQX {\fS(\balpha,\bzeta, \bxi; M,Q,X)}

\def\SAMJp{\cS_p(\balpha;\cM,\cJ)}
\def\SAMJq{\cS_q(\balpha;\cM,\cJ)}
\def\SAqMJq{\cS_q(\balpha_q;\cM,\cJ)}
\def\SAJq{\cS_q(\balpha;\cJ)}
\def\SAqJq{\cS_q(\balpha_q;\cJ)}
\def\SAIJp{\cS_p(\balpha;\cI,\cJ)}
\def\SAIJq{\cS_q(\balpha;\cI,\cJ)}

\def\RIJp{\cR_p(\cI,\cJ)}
\def\RIJq{\cR_q(\cI,\cJ)}

\def\TWXJp{\cT_p(\bomega;\cX,\cJ)}
\def\TWXJq{\cT_q(\bomega;\cX,\cJ)}
\def\TWpXJp{\cT_p(\bomega_p;\cX,\cJ)}
\def\TWqXJq{\cT_q(\bomega_q;\cX,\cJ)}
\def\TWJq{\cT_q(\bomega;\cJ)}
\def\TWqJq{\cT_q(\bomega_q;\cJ)}

 \def \xbar{\overline x}
  \def \ybar{\overline y}

\def\cA{{\mathcal A}}
\def\cB{{\mathcal B}}
\def\cC{{\mathcal C}}
\def\cD{{\mathcal D}}
\def\cE{{\mathcal E}}
\def\cF{{\mathcal F}}
\def\cG{{\mathcal G}}
\def\cH{{\mathcal H}}
\def\cI{{\mathcal I}}
\def\cJ{{\mathcal J}}
\def\cK{{\mathcal K}}
\def\cL{{\mathcal L}}
\def\cM{{\mathcal M}}
\def\cN{{\mathcal N}}
\def\cO{{\mathcal O}}
\def\cP{{\mathcal P}}
\def\cQ{{\mathcal Q}}
\def\cR{{\mathcal R}}
\def\cS{{\mathcal S}}
\def\cT{{\mathcal T}}
\def\cU{{\mathcal U}}
\def\cV{{\mathcal V}}
\def\cW{{\mathcal W}}
\def\cX{{\mathcal X}}
\def\cY{{\mathcal Y}}
\def\cZ{{\mathcal Z}}
\def\Ker{{\mathrm{Ker}}}

\def\NmQR{N(m;Q,R)}
\def\VmQR{\cV(m;Q,R)}

\def\Xm{\cX_m}

\def \A {{\mathbb A}}
\def \B {{\mathbb A}}
\def \C {{\mathbb C}}
\def \F {{\mathbb F}}
\def \G {{\mathbb G}}
\def \L {{\mathbb L}}
\def \K {{\mathbb K}}
\def \PP {{\mathbb P}}
\def \Q {{\mathbb Q}}
\def \R {{\mathbb R}}
\def \Z {{\mathbb Z}}
\def \fS{\mathfrak S}

\def\e{{\mathbf{\,e}}}
\def\ep{{\mathbf{\,e}}_p}
\def\eq{{\mathbf{\,e}}_q}

\def\\{\cr}
\def\({\left(}
\def\){\right)}
\def\fl#1{\left\lfloor#1\right\rfloor}
\def\rf#1{\left\lceil#1\right\rceil}

\def\Tr{{\mathrm{Tr}}}
\def\Nm{{\mathrm{Nm}}}
\def\Im{{\mathrm{Im}}}

\def \oF {\overline \F}

\newcommand{\pfrac}[2]{{\left(\frac{#1}{#2}\right)}}

\def \Prob{{\mathrm {}}}
\def\e{\mathbf{e}}
\def\ep{{\mathbf{\,e}}_p}
\def\epp{{\mathbf{\,e}}_{p^2}}
\def\em{{\mathbf{\,e}}_m}

\def\Res{\mathrm{Res}}
\def\Orb{\mathrm{Orb}}

\def\vec#1{\mathbf{#1}}
\def \va{\vec{a}}
\def \vb{\vec{b}}
\def \vn{\vec{n}}
\def \vu{\vec{u}}
\def \vv{\vec{v}}
\def \vz{\vec{z}}
\def\flp#1{{\left\langle#1\right\rangle}_p}
\def\T {\mathsf {T}}

\def\sfG {\mathsf {G}}
\def\sfK {\mathsf {K}}

\def\mand{\qquad\mbox{and}\qquad}

%\title[Additive energy of cyclic matrix groups]
%{Additive energy of cyclic matrix groups and character 
% sums with matrix exponential functions}

\title[Weil Sums over Small Subgroups]
{Weil Sums over Small Subgroups}

\author[A. Ostafe] {Alina Ostafe}
\address{School of Mathematics and Statistics, University of New South Wales, Sydney NSW 2052, Australia}
\email{alina.ostafe@unsw.edu.au}

\author[I. E. Shparlinski] {Igor E. Shparlinski}
\address{School of Mathematics and Statistics, University of New South Wales, Sydney NSW 2052, Australia}
\email{igor.shparlinski@unsw.edu.au}

 \author[J. F. Voloch] {Jos\'e Felipe Voloch}

\address{
School of Mathematics and Statistics,
University of Canterbury,
Private Bag 4800, Christchurch 8140, New Zealand}
\email{felipe.voloch@canterbury.ac.nz}

\begin{abstract}   We obtain new bounds on short Weil sums over small multiplicative subgroups of 
prime finite fields which remain nontrivial in the range the classical Weil bound is already trivial. 
The method we use is a blend of techniques coming from algebraic geometry and additive 
combinatorics. 
\end{abstract}



\maketitle
\tableofcontents
%
%
%

%---------------------------------------------------------------------
\section{Introduction}
\subsection{Set-up and motivation}
\label{sec:setup}

Let $p$ be a prime. Given a subset $\cX$ of a finite field $\F_p$ of $p$ elements and a polynomial $f \in \F_p[X]$,
we define the {\it Weil sum\/} over $\cX$ as
$$
S(\cX;f) = \sum_{x \in \cX} \ep(f(x)),
$$
where 
$$
\ep(z) = \exp(2 \pi i z/p),
$$
and we always assume that the elements of $\F_p$ are represented by the set $ \{0, \ldots, p-1\}$.

The celebrated result of Weil~\cite{Weil} asserts that for any nontrivial polynomial $f \in \F_p[X]$, when $\cX = \F_p$,
we have
\begin{equation}
\label{eq:Weil}
\left| S(\F_p;f) \right| \le (n - 1) p^{1/2},
\end{equation}
where $n = \deg f$, 
see also, for example,~\cite[Chapter~11]{IwKow} and~\cite[Chapter~6]{Li1}. 

The sums $S(\F_p;f)$ are usually called {\it complete sums\/}. The problem usually becomes 
harder for smaller sets $\cX$, that is, for sums called {\it incomplete sums\/}. 

Most of the attention the incomplete sums received is in the case of the sets $\cI_N = \{0, \ldots, N-1\}$ 
of $N \le p$ consecutive integers. In fact, using the classical Weil bound and the completing 
technique (see~\cite[Section~12.2]{IwKow}),  it is easy to show that in this case 
\begin{equation}
\label{eq:Weil Incompl}
S(\cI_N;f) = O(p^{1/2} \log p)
\end{equation}
for any nonlinear  polynomial $f \in \F_p[X]$, where the 
implied constant depends only on the degree $n$.  Clearly the bound~\eqref{eq:Weil Incompl} is nontrivial 
only for $N \ge p^{1/2+\varepsilon}$ for some fixed $\varepsilon>0$. 
For smaller values of $N$ one can also use general bounds
on the Weyl sums, see, for example,~\cite[Theorem~5]{Bourg4}, which remain nontrivial as long 
as $N \ge p^{1/n+\varepsilon}$,  which is an optimal range. 
In the special case of monomials $f(X) = X^n$, that is, for {\it Gauss sums\/}, Kerr~\cite{Kerr} obtained 
a better bound in the middle range of $N$. Kerr and Macourt~\cite[Theorem~1.4]{KeMa} also considered
exponential sums over {\it generalised arithmetic progressions\/} rather than over  intervals. 

The multiplicative analogues of this problem, when, instead of interval, the sum is over a multiplicative 
subgroup  $\cG\subseteq \F_p^*$ has also been studied, however significantly less general results are known.
Again, the Weil bound~\eqref{eq:Weil}, using that 
\begin{equation}
\label{eq:Weil subgr}
S(\cG;f) = \frac{\tau}{p-1} \sum_{x=1}^{p-1} \ep\(f\(x^{(p-1)/\tau}\)\) = O(p^{1/2}),
\end{equation}
instantly gives a nontrivial result for subgroups of order 
$\tau=\#\cG \ge p^{1/2+\varepsilon}$.   

In the case of linear polynomials, the bound of~\cite[Theorem~2]{Shp} 
has started a series of further improvements which goes beyond this limitation on $\#\cG$, see~\cite{Bourg3,BGK,DBGGGSST,HBK,Kon,Shkr} 
and references therein. 

Significantly less is known in the case of non-linear polynomials $f \in \F_p[X]$. 
Until very recently, the only known 
approach to such bounds is of Bourgain~\cite{Bourg1},  which actually works in a much more general scenario
of exponential sums with linear combinations of several exponential functions. This result of Bourgain~\cite{Bourg1} 
gives a bound saving some power $p^\eta$ compared to the trivial bound, however the exponent $\eta$ is not explicit 
and an attempt to make it explicit in~\cite[Theorems~4 and~5]{Pop} has some problem. 
Indeed, the argument in~\cite{Pop} quotes incorrectly the result of~\cite[Corollary~16]{Shkr3}, which,  
after correcting, leads to exponentially smaller saving. 

More recently, the authors~\cite{OSV} have used a different approach to a similar problem, based on a bound
for the number of rational points on curves over finite fields from~\cite{Vol}, which in some cases is 
stronger than the use of classical Weil bound ~\cite[Equation~(5.7)]{Lor}, and estimate Kloosterman sums 
$$
K(\cG; a, b) = \sum_{x \in \cG} \ep\(ax + bx^{-1}\) 
$$
over a subgroup $\cG$ of order $\tau$ with $a,b\in \F_p^*$. More precisely, 
the Weil bound, in form given for example, in~\cite[Theorem~2]{MorMor},  instantly 
gives 
$$
K(\cG; a, b)  = \frac{\tau}{p-1} \sum_{x=1}^{p-1} \ep\(ax^{(p-1)/\tau} + bx^{-(p-1)/\tau}\) 
  = O\(p^{1/2}\)
$$
for  $a,b\in \F_p^*$, which becomes trivial for $\tau < p^{1/2}$.   

On the other hand, by~\cite[Corollary 2.9]{OSV} we have 
\begin{equation}
\label{eq:Kloost}
K(\cG; a, b) \ll \tau^{20/27}   p^{1/9},
%%\min\{\tau^{23/36} p^{1/6},  \tau^{20/27}   p^{1/9} \},
\end{equation} 
where as usual the notations $U = O(V)$, $U \ll V$ and $ V\gg U$  
are equivalent to $|U|\leqslant c V$ for some positive constant $c$, 
which through out this work, all implied constants may depend only on $n$ (and
thus is absolute in~\eqref{eq:Kloost}).  Clearly, the bound~\eqref{eq:Kloost} is nontrivial
for $\tau\ge p^{3/7 + \varepsilon}$ for some fixed $\varepsilon>0$. 

We also note that several other bounds of exponential sums which rely on the result of~\cite{Vol}
and go beyond the Weil bound~\eqref{eq:Weil} 
have been given in~\cite{ShpVol}, see also~\cite{ShpWang}. 

 \subsection{Approach} 
It is important to recall that it has been shown in~\cite{Shp}, 
and used again in~\cite{HBK}, that in the case of linear polynomials, good bounds on the 4th moment
of the corresponding sums already allow us to  improve the Weil bound~\eqref{eq:Weil subgr}. 

However, 
for higher degree polynomials this is not sufficient and one needs strong bounds on at least the 6th moment. 
Hence, to obtain  nontrivial bounds 
for $S(\cG; f)$ for a given non-constant polynomial $f\in\F_p[X]$ and a small 
multiplicative  subgroup $\cG$ of $\F_p^*$,
we modify the ideas of our previous work~\cite{OSV} to investigate some high degree systems of 
polynomial equations. The main difficulty here is to study the {\it generic\/} absolute irreducibility of a certain 
family of curves and to be able to apply the result of~\cite[Theorem~(i)]{Vol}. We  then combine with the
 inductive approach of Bourgain~\cite{Bourg1}. 
%%  (in fact we rely on the exposition given by 
%% Garaev~\cite[Section~4.4]{Gar}).

%%As we have mentioned, here we combine some ideas from~\cite{OSV} with the inductive approach of Bourgain~\cite{Bourg1}
%%(in fact we rely on the exposition given by Garaev~\cite[Section~4.4]{Gar}) to obtain nontrivial bounds 
%%for $S(\cG; f)$ for given non-constant polynomial $f\in\F_p[X]$ and multiplicative  subgroup $\cG$ of $\F_p^*$.  

More precisely,  the method of Bourgain~\cite{Bourg1} (see also the exposition in~\cite[Section~4.4]{Gar}) is inductive on
 the number of non-constant terms $r$ of the polynomial, and it  requires the case $r=1$ and $r=2$ as the basis
 of induction. For $r=1$ we can use the bound of Shkredov~\cite{Shkr}, or even one of the
 earlier bounds from~\cite{HBK,Shp}. So we start with obtaining a bound for binomial sums over 
 a subgroup, see Lemma~\ref{lem:Binom}, which is similar to~\eqref{eq:Kloost}.   
 The proof of  Lemma~\ref{lem:Binom} resembles that of our previous work~\cite[Theorem~2.7]{OSV} but requires 
 to investigate the absolute irreducibility of some special polynomials. Then we use this 
 bound to initiate the induction and derive our main result. 
 
  %%\subsection{New bound} 
  \subsection{Main result}
For a real $\varepsilon>0$ we set 
\begin{equation}
\label{eq:eta2}
\eta_1(\varepsilon) = \eta_2(\varepsilon) = \frac{7}{27} \varepsilon,
 \end{equation}
 and define the sequence $\eta_n(\varepsilon)$, $n =3,4, \ldots$,
 recursively as follows
\begin{equation}
\label{eq:etan}
 \eta_n(\varepsilon) =  \frac{7\varepsilon}{18 \kappa_n(\varepsilon)},
 \end{equation} 
 where
 \begin{equation}
\label{eq:kappan}
 \kappa_n(\varepsilon) = \rf{ \frac{n-2 -7\varepsilon/3}{2 \eta_{n-1}(\varepsilon)} + 3}.
  \end{equation} 

\begin{thm}
\label{thm:Gen}
Let  $f(X) \in \F_p[X]$ be a  polynomial of degree $n\ge 1$, and let $\cG\subseteq \F_p^*$ 
be a subgroup of order $\tau \ge p^{3/7 + \varepsilon}$ for some fixed  $\varepsilon>0$.
Then 
$$
S(\cG; f) \ll  \tau p^{-\eta_n(\varepsilon)},
$$
where  $\eta_n(\varepsilon)$ is defined by~\eqref{eq:eta2} and~\eqref{eq:etan}.
\end{thm}

\begin{rem}
We have included the case of linear polynomials in Theorem~\ref{thm:Gen}, but as we have mentioned
in this case stronger results are available, see, for example,~\cite{Shkr}. 
\end{rem} 

\begin{rem}
It is obvious from our argument that if some information about the sparsity of the 
polynomial $f$ in  Theorem~\ref{thm:Gen} is known, then this can be accommodated 
in a stronger bound with $\eta_r(\varepsilon)$ instead of $\eta_n(\varepsilon)$
where $r\le n$ is the number of monomials in $f$.   
\end{rem} 

%%\subsection{Some comments}
Since it may not be easy to understand the behaviour of  the sequence 
$\eta_n(\varepsilon)$ from~\eqref{eq:etan} and~\eqref{eq:kappan}, 
here we give some clarifying examples.

First we compute explicitly,
$$
\eta_3(\varepsilon) =  \frac{7\varepsilon}{18  \rf{27 \varepsilon^{-1} /14 - 3/2}}.
$$

We also notice that a simple inductive argument shows that for, say, $\varepsilon < 1/2$, 
for some absolute constant $c > 0$ we have
$$
 \eta_n(\varepsilon) \ge c\frac{(7\varepsilon/9)^{n-1}}{(n-2)!}. 
 $$
 (certainly for $\varepsilon > 1/2$ the Weil bound~\eqref{eq:Weil subgr} is much stronger).

%%% Furthermore, it is well-known, see, for example~\cite{ErdMur}, that for a fixed integer $g \ne 0, \pm 1$, for 
%%% almost all primes $p$, in a sense of relative density, the subgroup generated by $g$ is of 
%%% order at least $p^{1/2}$. This makes the value of $\varepsilon = 1/14$ especially interesting.
%%% In this case we have
%%%\begin{table}[H]
%%%\centering
%%%\begin{tabular}{| c |  c| c |  c | c |  c | }
%%%\hline\hline
%%%$n$ & $2$  & $3$  & $4$  & $5$ \\ \hline
%%%$ \eta_n(1/14)$ & $2$  & $3$  & $4$  & $5$  \\ \hline\hline
%%%\end{tabular} 
%%%\caption{Some alues of $\eta_n(1/14)$}
%%%\label{Table}
%%%\end{table}
 
 \section{Algebraic geometry background}
\label{sec:AG} 
\subsection{Rational points on absolutely irreducible curves}

Let $q$ be a prime power. It is well-known that by the Weil bound we have 
 \begin{equation}
\label{eq:Weil curve}
\{(x,y)\in \F_q^2:~ F(x,y) = 0\} = q+ O\(d^2 q^{1/2}\)
\end{equation} 
for any absolutely irreducible polynomial  $F(X,Y) \in \F_q[X,Y]$ of degree $d$
(see, for example,~\cite[Section~X.5, Equation~(5.2)]{Lor}).  One can see that~\eqref{eq:Weil curve} is a genuine asymptotic formula only 
for $d = O(q^{1/4})$ and is in fact weaker than the trivial bound 
$$
\{(x,y)\in \F_q^2:~ F(x,y) = 0\} =  O\(d q\)
$$
for $d \ge q^{1/2}$, which is exactly the range of our interest. 
To obtain nontrivial bounds for such large values of $d$  we 
recall the following result, which is a {combination of~\cite[Theorem~(i)]{Vol} 
with the Weil bound~\eqref{eq:Weil curve} (and the trivial 
inequality  $p + 2d^2p^{1/2} \le 3p$ for $d \le p^{1/4}$). 

 \begin{lem}
\label{lem:HighDeg} Let $p$ be prime and let 
$F(X,Y) \in \F_p[X,Y]$ be an absolutely irreducible polynomial of degree $d<p$.
Then
$$
\# \{(x,y) \in \F_p^2:~F(x,y) = 0\} \le 4d^{4/3} p^{2/3} + 3p. 
$$
\end{lem}

\subsection{Absolute irreducibility of some polynomials}
To apply the bound of Lemma~\ref{lem:HighDeg}  we need to establish absolute irreducibility of polynomials relevant 
to our applications.  We present it in a general form for arbitrary finite fields, 
 as it may be useful for other applications.
Below we use a natural mapping of integers into elements of a finite field $\F_q$ of 
$q$ elements of characteristic $p$ via the reduction modulo $p$. 

\begin{lemma}
\label{lem:AbsIrred}
Given integers $n>m \ge 1$ with  $\gcd(m,n)=1$, there exists a non-zero polynomial  $\Delta(U,V)\in \Z[U,V]$, 
such that for every prime power $q$ and
positive integer $s$ with $\gcd(s,q)=1$ and $A,B \in \F_q$ with $\Delta(A,B) \ne 0$,
the polynomial
$$ 
F(X,Y) = (X^{sm} + Y^{sm} -A)^n - (X^{sn} + Y^{sn} -B)^m 
\in \F_q[X,Y]
$$
is absolutely irreducible.
\end{lemma}

\begin{proof} 
Let us begin by considering the case $s=1$.

We introduce a new variable $Z$ and note that for $s=1$ the curve $F=0$ is 
isomorphic to 
\begin{equation}
\label{eq:system}
  \begin{cases}
& X^{m} + Y^{m} -A = Z^m \\
&X^{n} + Y^{n} -B = Z^n. 
  \end{cases}
\end{equation}

The isomorphism is projection to the $X,Y$-plane, with inverse
given by 
$$(X,Y)\mapsto (X,Y,\(X^{m} + Y^{m} -A\)^u\(X^{n} + Y^{n} -B\)^v)$$ 
with some fixed integers $u$ and $v$ satisfying
$$
mu+nv=1.
$$

The two equations in~\eqref{eq:system} 
have gradients 
$$
m(X^{m-1},Y^{m-1},-Z^{m-1}) \mand n (X^{n-1},Y^{n-1},-Z^{n-1}),
$$
respectively. 
For the curve defined by the system~\eqref{eq:system}
 to be singular at a point $(x,y,z)$ the corresponding gradients have to be linearly
 dependent. This condition,
when fed back into~\eqref{eq:system} gives a relation between $A,B$. 

More explicitly we proceed as follows. 

First, we seek the polynomial $\Delta$ in the form 
\begin{equation}
\label{eq:Cond 1}
\Delta = mn \Delta_0
\end{equation}
with some $\Delta_0(U,V)\in \Z[U,V]$,  thus $\Delta(A,B) \ne 0$, guarantees that  $mn \ne 0$ in $\F_q$.

If $x=y=0$, then $z^m = -A, z^n = -B$, from~\eqref{eq:system} and so $(-A)^n = (-B)^m$.
Thus we also request 
\begin{equation}
\label{eq:Cond 2}
\((-A)^n - (-B)^m\) \mid \Delta(A,B).
\end{equation}

The possibilities $x=z=0$ and $y=z=0$ can be similarly treated and lead to the 
requirement 
\begin{equation}
\label{eq:Cond 3}
\(A^n - B^m\) \mid \Delta(A,B).
\end{equation}

If $x=0$ and  $yz \ne 0$, then the gradient condition 
gives $y = \zeta z$ with some $\zeta^{n-m}=1$
and~\eqref{eq:system} gives 
$$
\(\zeta^m-1\)z^m = A \mand   \(\zeta^n-1\)z^n = B.
$$
Therefore,
$\(\zeta^n-1\)^m A^n = \(\zeta^m-1\)^n B^m$, and we also request 
\begin{equation}
\label{eq:Cond 4}
\prod_{\zeta^{n-m}=1}  \(\(\zeta^n-1\)^m A^n - \(\zeta^m-1\)^n B^m\) \mid \Delta(A,B), 
\end{equation}  
where the product is taken  over all   roots of unity $\zeta$ with 
$\zeta^{n-m}=1$.

If $xyz \ne 0$, then we see from~\eqref{eq:grad}  that there has to be 
a constant $\lambda$ with %$x^{k-m} = y^{k-m} = z^{k-m}/c = \lambda$,
\begin{equation}
\label{eq:grad}
x^{n-m} = y^{n-m} = z^{n-m} = \lambda,
\end{equation}
so we can write $y = \zeta_1 x,$ and $z = \zeta_2 x$
with $\zeta_1^{n-m}= \zeta_2^{n-m}=1$,  leading to
\begin{equation}
\label{eq:Cond 5}
\prod_{\substack{\zeta_1^{n-m}= 1\\
\zeta_2^{n-m}=1}}  \(\(1+\zeta_1^n-\zeta_2^n\)^m A^n - \(1+\zeta_1^m-\zeta_2^m\)^n B^m  \) \mid \Delta(A,B),
\end{equation}
where the product is taken  over all pairs of roots of unity $(\zeta_1, \zeta_2)$ with 
$\zeta_1^{n-m}= \zeta_2^{n-m}=1$.

A similar argument works at infinity and shows that, for generic $A$ and $B$, 
%$A,B,C$ 
the curve is smooth.

Given $X,Y$, there is a unique choice of $Z$ satisfying~\eqref{eq:system},
so the projection to the $X,Y$ does not acquire singularities from 
distinct points in three-space. The only singularities are cusps coming
from a vertical tangent line which are unibranched (since the curve in three-space is smooth). However, a reducible
plane curve has singular points with more than one branch wherever two components
meet. Hence~\eqref{eq:system} is an irreducible curve. (See~\cite[Chapter~16]{Kunz} for a detailed exposition
of branches of curve singularities).

We have shown that, for $s=1$, the polynomial $F$ is absolutely irreducible. We consider the algebraic curve $\cC$
which is a non-singular projective model of $F=0$ (still with $s=1$). 

Suppose $n>m>1$. We now use an argument similar to~\cite[Lemma~4.3]{OSV}.

For a point $P=(0,y_0)$ on the curve $F=0$ we have 
$$\frac{\partial F}{\partial X} (0,y_0) = 0.
$$
Next we show that the point $P=(0,y_0)$ is a simple point on the curve $F=0$ with 
\begin{equation}
\label{eq:Nonvanish}
\frac{\partial F}{\partial Y}(0,y) \ne 0.
\end{equation}
(for generic $A$ and $B$).
It now suffices to show that the discriminant $D(A,B)$ of $F(0,Y)$ (as a polynomial in $A,B$) is
not identically zero. 

Taking $A=1$ and $B = 0$,  the polynomial $F$ specialises to 
$$
(Y^m-1)^n - Y^{mn} = \prod_{\xi^n = 1} (Y^m(1-\xi) -1),
$$ which has a non-zero discriminant as
each factor $(Y^m(1-\xi) -1)$ is square-free and these factors are relatively prime. 
Thus we impose the condition 
\begin{equation}
\label{eq:Cond 6}
D(A,B) \mid \Delta(A,B).
\end{equation}

It remains to choose $ \Delta(A,B) \in \Z[A,B]$ as an arbitrary fixed polynomial
which depends only on $m$ and $n$ and satisfies the divibility conditions~\eqref{eq:Cond 1}, \eqref{eq:Cond 2},
\eqref{eq:Cond 3}, \eqref{eq:Cond 4}, \eqref{eq:Cond 5}
and~\eqref{eq:Cond 6}.

%Indeed the conditions are:
%$$(y^k-A)^m - (y^m-B)^k = mk((y^k-A)^{m-1}y^{k-1} - (y^m-B)^{k-1}y^{m-1})=0$$
%and, in the presence of the first condition, the second condition simplifies to

%$$(y^k-A)^m\left(\frac{A}{(y^k-A)} - \frac{B}{(y^m-B)}\right)=0$$

We now consider the case of arbitrary $s\ge 1$ (and $n>m>1$).

So $P$ corresponds to a place of $\cC$. We consider the functions $x,y$ on $\cC$ that satisfy the equation $F(x,y)=0$. The function $x$ has a simple zero at $P$, hence is not a power of another function on $\cC$.
It follows from~\cite[Proposition~3.7.3]{Stichtenoth}, that the equation $U^s = x$ is irreducible over the function field of $\cC$
and defines a cover $\cD$ of $\cC$. Now, consider any point $Q$ on $\cD$ above a point $(x_0,0)$ on $\cC$. Since $x$ is not zero at $(x_0,0)$ (for generic $A,B$), %(for generic $A,B,C$),
the curve $\cD$ is locally isomorphic to $\cC$ near $Q$ and we conclude, as above, that the function $y$ on $\cD$ has a simple zero at $Q$
and, in particular, is not a power of another function on $\cD$. Again, we conclude that 
the equation $W^s = y$ is irreducible over the function field of $\cD$ and defines a cover $\cE$ of $\cD$. 
In other words, $F(U^s,W^s)=0$ is an absolutely irreducible equation defining the curve $\cE$, 
which concludes the case $n>m>1$.  

We are now  left with $n>m=1$.
It is still true that a point $P=(0,y_0)$ is a simple point on the curve $F=0$ with~\eqref{eq:Nonvanish} 
(for generic $A,B$). %(for generic $A,B,C$). 
Indeed, if
$$
0 = \frac{\partial F}{\partial Y}(0,y)  (0,y_0) = n(y_0^{n-1} -(y_0-B)^{n-1}), 
$$
then combining this with
$$
0=F(0,y_0) = -A+y_0^{n} -(y_0-B)^{n} ,
$$
we derive
$$
0=  -A+y_0^{n} -(y_0-B)^{n} =  -A+y_0^{n} -(y_0-B) y_0^{n-1} =    -A+B y_0^{n-1} .
$$
Hence 
$$
0 = By_0^{n-1} -B(y_0-B)^{n-1} =  A -B(y_0-B)^{n-1}.
$$
Considering the resultant of $B(Y-B)^{n-1} -A$ and $BY^{n-1} - A$  (which clear does not vanish 
for $A=1$ and $B=0$ and thus is a nontrivial polynomial) 
gives a contradiction for generic $A,B$.  
The proof then continues as before in the case $n>m>1$. 
\end{proof}

 \section{Exponential sums and systems of diagonal equations}

 \subsection{Exponential sums and the number of solutions to some systems of equations}
 \label{sec:Exp2Eq}
 Here we collect some previous results on exponential sums and 
 also about links between these bounds and the number of solutions to 
 some congruences.

Given an integer vector $\vn = \(n_1, \ldots, n_r\)\in \Z^r$ 
 with $n_r> \ldots > n_1 \ge 1$, and a subgroup $\cG \subseteq \F_p^*$, 
we denote by   $Q_{k}(\vn;\cG)$
the number of solutions to the 
following system of $r$ equations
\begin{equation}
\label{eq:gen syst}
\begin{split}
 g_1^{n_i}+\ldots + g_k^{n_i} & = g_{k+1}^{n_i}+\ldots + g_{2k}^{n_i}, \qquad i =1, \ldots, r, \\
& g_1, \ldots ,  g_{2k} \in \cG.
\end{split}
 \end{equation}
%% 
%%In Section~\ref{sec:bin sum} we linked bounds $Q_{2}(m,n; \cG)$ and $Q_{3}(m,n; \cG)$ 
%%to bounds on binomial sums. Here we reverse this links and use bounds 
%%of exponential sums to estimate $Q_{k}(\vn;G)$ for large values of $k$ (and
%%also for higher dimensional vectors $\vn$.  

The following link between $S(\cG; f)$  and  $Q_{k}(\vn;G)$ is a slight variation 
of several previous results of a similar spirit.

  \begin{lem}
\label{lem:S and Q} Let
$$
 f(X)  = a_r X^{n_r} + \ldots + a_1 X^{n_1}  \in \F_p[X]
 $$
 with nonzero coefficients $a_1, \ldots, a_r \in \F_p^*$ and integer exponents $n_r> \ldots > n_1 \ge 1$.
Then, for  a subgroup $\cG \subseteq \F_p^*$ and for any positive integers $k$ and $\ell$, we have 
$$
\left|S(\cG; f)\right|^{2k\ell} \le p^r \tau^{2k\ell - 2k -2 \ell} Q_{k}(\vn;\cG) Q_{\ell}(\vn;\cG).
$$
\end{lem}
\begin{proof}
We start by noticing that, for any $h \in \cG$, we have
 $$
S(\cG;f)=  \sum_{g  \in \cG}   \ep\(a_1 (hg)^{n_1} + \ldots + a_r (hg)^{n_r}\).
$$
Hence, for any integer $k \ge 1$ we have
\begin{align*}
\tau\(S(\cG;f)\)^k &=  \sum_{h  \in \cG} \(\sum_{g  \in \cG}   \ep\(a_1 (hg)^{n_1} + \ldots + a_r (hg)^{n_r}\)\)^k\\
& = \sum_{h  \in \cG} \sum_{\lambda_1, \ldots, \lambda_r \in \F_p}   J_k(\lambda_1,  \ldots, \lambda_r)
  \ep\(\lambda_1 h^{n_1} + \ldots + \lambda_r h^{n_r}\), 
  \end{align*}
 where $J_k(\lambda_1,  \ldots, \lambda_r)$ is the number of solutions to the following system of 
 equations:
 \begin{align*}
a_i\( g_1^{n_i}+\ldots + g_k^{n_i}\)& =  \lambda_i, \qquad i =1, \ldots, r, \\
 g_1&, \ldots ,  g_{k} \in \cG.
\end{align*}
Hence, changing the order of summations, we obtain
$$
\tau\left|S(\cG;f)\right|^k   \le  \sum_{\lambda_1, \ldots, \lambda_r \in \F_p}   J_k(\lambda_1,  \ldots, \lambda_r)
 \left| \sum_{h  \in \cG} \ep\(\lambda_1 h^{n_1} + \ldots + \lambda_r h^{n_r}\) \right|. 
$$
Observe that since $a_1, \ldots, a_r \ne 0$, we have 
 \begin{align*}
& \sum_{\lambda_1, \ldots, \lambda_r \in \F_p}   J_k(\lambda_1,  \ldots, \lambda_r) = \tau^k, \\
& \sum_{\lambda_1, \ldots, \lambda_r \in \F_p}   J_k(\lambda_1,  \ldots, \lambda_r)^2 = 
Q_{k}(\vn;G), 
\end{align*}
where $\vn = \(n_1, \ldots, n_r\)$.

Writing 
$$
 J_k(\lambda_1,  \ldots, \lambda_r) =  J_k(\lambda_1,  \ldots, \lambda_r)^{1-1/\ell} \(J_k(\lambda_1,  \ldots, \lambda_r)^2\)^{1/2\ell} 
 $$
 and applying the H{\"o}lder inequality, we derive 
 \begin{align*}
\tau^{2\ell}\(S(\cG;f)\)^{2k\ell} & \le \( \sum_{\lambda_1, \ldots, \lambda_r \in \F_p}   J_k(\lambda_1,  \ldots, \lambda_r) \)^{2\ell-2}\\
& \qquad  \quad 
 \sum_{\lambda_1, \ldots, \lambda_r \in \F_p}   J_k(\lambda_1,  \ldots, \lambda_r)^2  \\
  &  \qquad  \qquad  \quad  \sum_{\lambda_1, \ldots, \lambda_r \in \F_p}  
 \left| \sum_{h  \in \cG} \ep\(\lambda_1 h^{n_1} + \ldots + \lambda_r h^{n_r}\) \right|^{2\ell}\\
 & = p^{r} \tau^{k(2\ell-2)} Q_{k}(\vn;G) Q_{\ell}(\vn;G)
  \end{align*}
  which concludes the proof. 
\end{proof}

We now establish a link in the opposite direction, that is, from bounds on exponential sums 
to bounds on $ Q_{k}(\vn; \cG)$. 

 \begin{lemma}
\label{lem:Cond Bound} Let $r \ge 2$ and let $\varepsilon\ge 0$ be fixed.
  Assume that there is some  fixed $\eta> 0$ (depending only on $r$ and $\varepsilon$) such that for 
all nonzero vectors $\(a_1, \ldots, a_r\)\in \F_p^r$ and for a vector $\vn = \(n_1, \ldots, n_r\)\in \Z^r$ with $n_r> \ldots > n_1 \ge 1$,  for a subgroup $\cG \subseteq \F_p^*$ of order $\tau$
with 
\begin{equation}
\label{eq:medium tau}
 p^{3/7+\varepsilon} \le \tau \le p^{3/4}
 \end{equation}
 we have 
\begin{equation}
\label{eq:Small sum}
 \sum_{g \in \cG} \ep\(a_1 g^{n_1} + \ldots + a_r g^{n_r}\) \ll \tau p^{-\eta}.
 \end{equation}
 Then for any integer $k\ge 3$  we have 
$$
 Q_{k}(\vn; \cG) \ll \tau^{2k} p^{-\xi} 
 $$
 where $\xi = \min\{r, \eta(2k-6) +1 +7\varepsilon/3\}$. 
%%Then for an integer $k$ such that 
%%$$
%%k \ge    \frac{1}{2}  \(r-1 +7\varepsilon/3\) \eta^{-1} + 3
%%$$
%%we have 
%%$$
%% Q_{k}(\vn; \cG) \ll \frac{\tau^{2k}}{p^r}.
%% $$
\end{lemma}

\begin{proof}
Using the orthogonality of characters, we write 
\begin{equation}
\label{eq:Qkmn}
\begin{split}
 Q_{k}(\vn&, \cG)\\
 & = \frac{1}{p^r}  \sum_{a_1, \ldots, a_r \in \F_p} \left|\sum_{g  \in \cG}  \ep\(a_1 g^{n_1} + \ldots + a_r g^{n_r}\) \right|^{2k}\\
 & = \frac{\tau^{2k}}{p^r} +  \frac{1}{p^r}  \sum_{\substack{a_1, \ldots, a_r\in \F_p\\ \(a_1, \ldots, a_r\)\  \ne \mathbf{0}}}
  \left|\sum_{g  \in \cG}   \ep\(a_1 g^{n_1} + \ldots + a_r g^{n_r}\) \right|^{2k}.
\end{split}
 \end{equation}
 
Now, using our assumption~\eqref{eq:Small sum}  we obtain 
\begin{align*}
  \frac{1}{p^r}  \sum_{\substack{a_1, \ldots, a_r\in \F_p\\ \(a_1, \ldots, a_r\)\  \ne \mathbf{0}}} &
  \left|\sum_{g  \in \cG}   \ep\(a_1 g^{n_1} + \ldots + a_r g^{n_r}\) \right|^{2k} \\
  &\ll
  \frac{ \(\tau   p^{-\eta} \)^{2k-6}}{p^r}  \sum_{\substack{a_1, \ldots, a_r\in \F_p\\ \(a_1, \ldots, a_r\)\  \ne \mathbf{0}}}
  \left|\sum_{g  \in \cG}   \ep\(a_1 g^{n_1} + \ldots + a_r g^{n_r}\) \right|^{6} .
  \end{align*}
  Dropping the restriction $\(a_1, \ldots, a_r\)\  \ne \mathbf{0}$ from the summation, we now obtain
  \begin{align*}
  \frac{1}{p^r}  \sum_{\substack{a_1, \ldots, a_r\in \F_p\\ \(a_1, \ldots, a_r\)\  \ne \mathbf{0}}} &
  \left|\sum_{g  \in \cG}   \ep\(a_1 g^{n_1} + \ldots + a_r g^{n_r}\) \right|^{2k} \\
  &\ll
  \frac{ \(\tau   p^{-\eta} \)^{2k-6}}{p^r}  \sum_{a_1, \ldots, a_r\in \F_p}
  \left|\sum_{g  \in \cG}   \ep\(a_1 g^{n_1} + \ldots + a_r g^{n_r}\) \right|^{6} \\
  & =  \(\tau   p^{-\eta} \)^{2k-6}   Q_{3}(\vn; \cG) .
  \end{align*}
  Since $r \ge 2$, we obviously have 
  $$
  Q_{3}(\vn; \cG) \le   Q_{3}(n_1,n_2; \cG) .
  $$
Thus applying  Corollary~\ref{cor:Q3} in  Section~\ref{sec;syst eq}  below    
and  using that under our assumption~\eqref{eq:medium tau} 
we have $\tau^{11/3} \ge  \tau^5/p$, 
we obtain 
  \begin{align*}
 \frac{1}{p^r}  \sum_{\substack{a_1, \ldots, a_r\in \F_p\\ \(a_1, \ldots, a_r\)\  \ne \mathbf{0}}} 
  \left|\sum_{g  \in \cG}   \ep\(a_1 g^{n_1} + \ldots + a_r g^{n_r}\) \right|^{2k}
 & \ll  \(\tau   p^{-\eta} \)^{2k-6} \tau^{11/3} \\
 & = \tau^{2k-7/3} p^{-\eta(2k-6)}. 
  \end{align*}
  Recalling~\eqref{eq:medium tau} again we see that 
\begin{align*}
 \frac{1}{p^r}  \sum_{\substack{a_1, \ldots, a_r\in \F_p\\ \(a_1, \ldots, a_r\)\  \ne \mathbf{0}}} &
  \left|\sum_{g  \in \cG}   \ep\(a_1 g^{n_1} + \ldots + a_r g^{n_r}\) \right|^{2k}\\
 & \qquad  \ll   \tau^{2k} p^{-\eta(2k-6) -1 -7\varepsilon/3}, 
 %% \ll \tau^{2k} p^{-r}, 
 \end{align*}
%% provided 
% $$
% \eta(2k-6) +1 + 7\varepsilon/3 \ge r
% $$
 which together with~\eqref{eq:Qkmn} concludes the proof. 
\end{proof}

\subsection{Bounds on the number of solutions to some systems of  equations in six variables}
\label{sec;syst eq}
We start with an observation that the results of this section are independent of those 
in Section~\ref{sec:Exp2Eq} and hence there is no logical problem in our use of them 
in the proof of Lemma~\ref{lem:Cond Bound}.

For $r=2$ and $\vn = (m,n)$ we write $Q_{k}(m,n;\cG)$ for $Q_{k}(\vn;\cG)$.

Here we obtain some bounds on  $Q_{3}(m,n;\cG)$.  In fact, it is easier to
work with  the following system of equations
\begin{equation}
\label{eq: xkmns}
  \begin{cases}
&x_1^{sm}+x_2^{sm}+ x_3^{sm} =  x_{4}^{sm}+x_5^{sm} + x_{6}^{sm}\\
& x_1^{sn}+x_2^{sn}+ x_3^{sn} =  x_{4}^{sn}+x_5^{sn} + x_{6}^{sn}
  \end{cases}, \quad 
x_1, \ldots ,  x_{6} \in \F_p^*,
\end{equation}
instead of the system of the type~\eqref{eq:gen syst} with group elements.

Denoting by $T_{3}(m,n;s)$ the number of solutions to~\eqref{eq: xkmns} we see that 
\begin{equation}
\label{eq: Q and T} 
Q_{3}(m,n;\cG) = s^{-6} T_{3}(m,n;s), 
\end{equation}
where 
$$
s = \frac{p-1}{\tau}
$$
and as before $\tau = \# \cG$. 

\begin{lemma}
\label{lem:T3} 
For integers $n > m >0$, we have
$$
 T_{3}(m,n;s) \ll s^{7/3} p^{11/3} + sp^4.
 $$ 
  \end{lemma}

\begin{proof} 
First we note that if $\gcd(m,n)= d$ then 
$$ 
T_{3}(m,n;s) \le  e^6 T_{3}(m/d,n/d;s), 
$$
where $e = \gcd(d,p-1)$. 

Hence we can assume that 
\begin{equation}
\label{eq: gcd 1} 
\gcd(m,n)=1, 
\end{equation}
which enables us to apply Lemma~\ref{lem:AbsIrred}.

We  now fix $x_4$, $x_5$ and $x_6$ and thus we obtain  
$(p-1)^3$ systems of equations of the form
$$
 \begin{cases}
&x_1^{sm}+ x_2^{sm}+ x_3^{sm} = A\\
& x_1^{sn}  + x_2^{sn}+ + x_3^{sn}   = B  \end{cases}, \qquad 
x_1,  x_2, x_3 \in \F_p^*, 
$$
where
$$
A=x_4^{sm}+ x_5^{sm}+ x_6^{sm} \mand B=x_4^{sn}+ x_5^{sn}+ x_6^{sn},
$$ 
from which we derive 
\begin{equation}
\label{eq: xxsm xxsn} 
\(x_1^{sm}+ x_2^{sm}  -A\)^n = \(x_1^{sn} + x_2^{sn} -B\)^m.
\end{equation}

Under the assumption~\eqref{eq: gcd 1}, let the polynomials 
 $\Delta(U,V)\in \Z[U,V]$ be as in Lemma~\ref{lem:AbsIrred}. 
 
 Since $\Delta$ depends only on $m$ and $n$, we see that if $p$ is large 
 enough, $\Delta$ is a non-zero polynomial modulo $p$.

We assume first that $\Delta(A,B) = 0$ for  a pair $(A,B)\in \F_p^2$ as above.  
 Thus
\begin{equation}
\label{eq:delta eq}
 \Delta\(x_4^{sm}+ x_5^{sm}+ x_6^{sm}, x_4^{sn}+ x_5^{sn}+ x_6^{sn}\) = 0.
\end{equation}

If $\Delta\(X^{sm}+x_5^{sm}+ x_6^{sm},X^{sn}+ x_5^{sn}+ x_6^{sn}\)$, as a polynomial in $X$, is not identically zero for some $(x_5,x_6)\in\F_p^2$, then obviously it has $O(s)$ zeros. Thus, in this case, the equation~\eqref{eq:delta eq} has  $O(sp^2)$ solutions $(x_4,x_5,x_6)\in\F_p^3$.
 
 On the other hand, if $\Delta\(X^{sm}+x_5^{sm}+ x_6^{sm},X^{sn}+ x_5^{sn}+ x_6^{sn}\)$, as a polynomial in $X$, is identically zero, then it also holds for $X=0$, thus
\begin{equation}
\label{eq: Bad x4x5} 
 \Delta\(x_5^{sm}+ x_6^{sm}, x_5^{sn}+ x_6^{sn}\)=0.
 \end{equation}
Now, a similar argument shows that~\eqref{eq: Bad x4x5}  holds for $O(sp)$ pairs $(x_5,x_6)$
 for which there are at most $p$ values of $x_4$. 
 
 Therefore, the equation~\eqref{eq:delta eq} has   $O(sp^2)$ solutions $(x_4,x_5,x_6)\in\F_p^3$
  in total.

For each of such $O\(sp^2\)$ values of $(x_4,x_5,x_6)$ the 
corresponding equation~\eqref{eq: xxsm xxsn}  is 
 nontrivial  since it contains a unique term $nx_1^{sm(n-1)} x_2^{sm}$
and hence has $O(sp)$ solutions $(x_1,x_2)$ after which there are $O(s)$ possible values for $x_3$.
Hence, the total contribution from the case $\Delta(A,B) = 0$ is  $O\(s^3 p^3\)$. 

If  $\Delta(A,B) \ne 0$, then for the corresponding $O(p^3)$ possibilities for $(x_4, x_5, x_6) \in \F_p^3$, 
 by Lemma~\ref{lem:AbsIrred},  we can apply 
Lemma~\ref{lem:HighDeg} to bound the number of solutions to~\eqref{eq: xxsm xxsn} (after which we have $O(s)$ possibilities 
for $x_3$). 
 Hence, the total contribution from the case 
$\Delta(A,B) \ne 0$ is  $O\(s \(s^{4/3} p^{2/3} + p\)p^3\)$. 

Therefore $T_{3}(m,n;s) \le s^3 p^3 + s^{7/3} p^{11/3} + sp^4$.
Since $s^3 p^3 \le s^{7/3} p^{11/3}$ for $s \le p$,   the result follows.  
\end{proof} 

Recalling~\eqref{eq: Q and T}, we see that  Lemma~\ref{lem:T3} implies
the following. 

\begin{cor}
\label{cor:Q3} 
For integers $n > m >0$, we have
$$
Q_{3}(m,n; \cG) \ll \tau^{11/3} + \tau^5/p.
 $$
  \end{cor}
  
  \begin{rem}
We recall that the method of Kurlberg and Rudnick~\cite[Lemma~5]{KR},
immediately implies that $Q_{2}(m,n; \cG) \ll \tau^{2}$. However this bound is
not sufficient for our purpose. 
\end{rem}

 \subsection{Bounds on monomial and  binomial sums}
 \label{sec:bin sum}
 
  First we recall the following result of Shkredov~\cite[Theorem~1]{Shkr} (with a slight 
 generalisation and also combined with a direct implication of~\eqref{eq:Weil}).
 
  \begin{lem}
\label{lem:Monom}
Let  $f(X)  = aX^n \in \F_p[X]$ of degree $n\ge 1$ and with $a  \ne 0$,  and let $\cG\subseteq \F_p^*$ 
be a   subgroup $\cG$ of order $\tau$. Then 
$$
S(\cG; f) \ll   \min\{p^{1/2}, \tau^{1/2}   p^{1/6} (\log p)^{1/6}\}.
$$
\end{lem}

\begin{proof}
We remark that the result of  Shkredov~\cite[Theorem~1]{Shkr} corresponds to $n=1$. Otherwise we note that 
$$
S(\cG; f) = d  \sum_{x \in \cG^d} \ep(ax)
$$
where $d = \gcd(\tau, n)$ and $\cG^d= \{g^d:~ g \in \cG\}$. 
\end{proof}

  We now derive the following estimate,  which improves~\eqref{eq:Weil subgr}
  for $\tau \le p^{21/40}$, remains nontrivial for $\tau \ge p^{3/7+\varepsilon}$ for any fixed $\varepsilon>0$
   and  which we believe is of independent interest.
For this, we apply Lemma~\ref{lem:S and Q} with  $k=\ell=3$ and we use Corollary~\ref{cor:Q3}.
  
 \begin{lemma}
\label{lem:Binom}
Let  $f(X)  = aX^m + bX^n \in \F_p[X]$ with integers $n > m \ge 1$ and $(a,b) \ne (0,0)$, and let $\cG\subseteq \F_p^*$ 
be a   subgroup of order $\tau$. Then 
$$
S(\cG; f) %% \ll \min\{p^{1/2}, \tau^{23/36} p^{1/6},
\ll   \tau^{20/27}   p^{1/9}.
  $$
\end{lemma}

\begin{proof}
 If $ab= 0$ then the result is instant from Lemma~\ref{lem:Monom}
 
 Hence we now assume $a,b \in \F_p^*$. We apply Lemma~\ref{lem:S and Q} with 
 $$
 (k, \ell) = (3,3)
 $$
and the bound of Corollary~\ref{cor:Q3}.
 We also note that for $\tau >p^{21/40}$ we have 
$$
p^{1/2} \le   \tau^{20/27}   p^{1/9}.
 $$
 Hence we only need to apply  Corollary~\ref{cor:Q3} for $\tau \le p^{21/40}$ 
 in which case $ \tau^{11/3} \ge  \tau^5/p$,  and thus in this case  we simply have 
 $Q_{3}(m,n; \cG) \ll \tau^{11/3}$ and the result follows. 
\end{proof}

 %%\subsection{Bounding  the number of solutions to multivariate system of diagonal equations via exponential sums} 

  \section{Proof of Theorem~\ref{thm:Gen}} 
  \subsection{Preliminaries and the basis of induction} 
  We prove the result by induction on the number of terms $r$ in the polynomial
$$
 f(X)  = a_r X^{n_r} + \ldots + a_1 X^{n_1}  \in \F_p[X]
 $$
 with nonzero coefficients $a_1, \ldots, a_r \in \F_p^*$ and integer exponents $n =n_r> \ldots > n_1 \ge 1$.
 
We  see from Lemma~\ref{lem:Binom}  that for $r=1,2$ the condition~\eqref{eq:Small sum} 
of Lemma~\ref{lem:Cond Bound}
is satisfied with 
$$
\eta = \eta_1(\varepsilon) \mand \eta = \eta_2(\varepsilon),
$$
respectively, 
where $\eta_{1}(\varepsilon)$ and $\eta_{2}(\varepsilon)$ are given by~\eqref{eq:eta2}, which form the basis of induction. 

  \subsection{Inductive step} 
  Assume that that the result holds for all nontrivial polynomials of degree at most $n$ with at most 
  $r-1$ monomials and we prove it for polynomials with $r \ge 3$ monomials. 
   First we note that we can assume that $\tau \le p^{3/4}$ since otherwise 
  the bound~\eqref{eq:Weil subgr} is stronger than that of Theorem~\ref{thm:Gen}.
  
  In particular, the condition~\eqref{eq:medium tau} of Lemma~\ref{lem:Cond Bound}  
  is satisfied. 
  We fix some arbitrary positive integers $k$, $\ell$, $u$ and $v$ with $u, v \le r$.
  Let 
  $$
  \vn_u = (n_1, \ldots, n_u) \mand \vn_v= (n_1, \ldots, n_v) .
  $$
  We now use the trivial
  bounds
\begin{equation}
\label{eq:reduction}
  Q_{k}(\vn;\cG) \le  Q_{k}(\vn_u;\cG) \mand Q_{\ell}(\vn;\cG)\le Q_{\ell}(\vn_v;\cG).
\end{equation}
 In fact we choose $u = r-1$ and $v  = 2$. 
 Furthermore, we recall the definition~\eqref{eq:kappan} and set 
\begin{equation}
\label{eq:k l}
 k = \kappa_r(\varepsilon) \mand \ell = 3,
\end{equation}
 in which case, using the induction assumption,
  by Lemma~\ref{lem:Cond Bound}, used with $u = r-1$ instead of $r$, we have 
\begin{equation}
\label{eq:Qknr-1}
 Q_{k}(\vn_{r-1};\cG)\ll \tau^{2k}/p^{r-1},
\end{equation}
 while by Corollary~\ref{cor:Q3}, using that $\tau\le p^{3/4}$, we obtain  
\begin{equation}
\label{eq:Q3n2}
 Q_{3}(\vn_2;\cG) \ll  \tau^{11/3} +  \tau^5/p \ll  \tau^{11/3}.
\end{equation}
Indeed,~\eqref{eq:Qknr-1} follows from the definition of $\kappa_r(\varepsilon)$ in~\eqref{eq:kappan}, which ensures that $\xi =r-1$ in Lemma~\ref{lem:Cond Bound}.

Substituting the bounds~\eqref{eq:reduction}, \eqref{eq:Qknr-1} and \eqref{eq:Q3n2} in the bound of Lemma~\ref{lem:S and Q}  
we obtain 
 $$
S(\cG; f)^{6k} \ll p^r \tau^{4k -6} \frac{\tau^{2k}}{p^{r-1}} \tau^{11/3}
= \tau^{6k -7/3} p \le  \tau^{6k -7\varepsilon/3} .
$$
Hence 
 $$
S(\cG; f) \ll \tau^{1 -7\varepsilon/18k}.
$$
Recalling the definition~\eqref{eq:etan} and the choice of $k$ in~\eqref{eq:k l}, we  conclude the proof.

\section{Comments}

If $\vartheta$ is a generator of $\cG$ then the sum $S(\cG;f)$ can be written as
$$
S(\cG;f) =  \sum_{x=1}^{\tau} \ep\(f\(\vartheta^{x}\)\).
$$
This reformulations also allows to generalise these sums  to twisted sums, 
$$
S_b(\cG;f) =  \sum_{x=1}^{\tau} \ep\(f\(\vartheta^{x}\)\)   \exp(2 \pi i bx/\tau),
$$
to which all our results apply without any changes (with just minor typographic 
adjustments). In turn, together with the well-know completing technique (see, for example,~\cite[Section~12.2]{IwKow}) 
bounds on the sums $S_b(\cG;f) $ lead to bounds on incomplete sums 
$$
 \sum_{x=1}^{N} \ep\(f\(\vartheta^{x}\)\), \qquad  1 \le N \le \tau.
 $$
 We note that in~\cite{NiWi} sequences of the form $\(f\(\vartheta^{x}\)\)$ have been studied
 as sources of pseudorandom numbers, but with nontrivial results only in the case of 
 periods $\tau > p^{1/2+\varepsilon}$,  while our results allows us to extend this range to 
$\tau > p^{3/7+\varepsilon}$.

We also note that in~\cite{MeWi,Wint} the sequence
$$
\(a\vartheta^{x}+b\)^{-1}, \qquad n =1, 2, \ldots,
$$
has been suggested as a source of pseudorandom numbers. Unfortunately neither the method of 
Bourgain~\cite{Bourg1} nor of this work applies to the corresponding exponential sums 
$$
 \sum_{x=1}^{N} \ep\(\(a\vartheta^{x}+b\)^{-1}\), \qquad  1 \le N \le \tau, 
 $$
 (with a natural convention that the values with $a\vartheta^{x}=-b$ are excluded), 
which are necessary for investigating this sequence.  
So we leave a question of obtaining such nontrivial  bounds  for $\tau < p^{1/2}$ as 
an open problem.  Even the case of complete sums  
$$
\sum_{g \in \cG}  \ep\(\(ag+b\)^{-1}\) 
$$
is of interest.
 
\section*{Acknowledgement}

During the preparation of this work, the first two authors (A.O. and I.E.S.) were partially supported by the
Australian Research Council Grant DP200100355. The third author (J.F.V.) was partially supported by a grant from the Ministry of Business, Innovation and Employment. %  and the Marsden Fund Council administered by the Royal Society of New Zealand.

\end{document}